\documentclass[12pt, reqno]{amsart}

\usepackage[english]{babel}
\usepackage[left=3.4cm, right=3.4cm, top=3.3cm, bottom=3.9cm]{geometry}
\usepackage{url, amsfonts, amsmath, amsthm, amssymb, mathtools, mathrsfs, enumerate}
\usepackage{color, xcolor, verbatim, tensor, tikz, tikz-cd}
\usetikzlibrary{angles,quotes}
\usepackage[colorlinks=true,citecolor=blue]{hyperref}
\usetikzlibrary{matrix,calc}
\definecolor{wine-stain}{rgb}{0.5,0,0}
\newcommand{\red}[1]{\textcolor{red}{#1}}
\newtheorem{thm}{Theorem}[section]
\newtheorem{prop}[thm]{Proposition}

\newtheorem{cor}[thm]{Corollary}

\theoremstyle{definition}

\newtheorem{rem}[thm]{Remark}

\newtheorem*{ques*}{Question}
\newtheorem*{thm*}{Theorem}
\newtheorem*{rem*}{Remark}
\newtheorem*{rems*}{Remarks}
\newtheorem*{exs*}{Examples}
\newtheorem*{mthm*}{Main Theorem}
\numberwithin{equation}{section}
\newcommand{\al}{\alpha}
\newcommand{\be}{\beta}
\newcommand{\ga}{\gamma}
\newcommand{\del}{\delta}
\newcommand{\la}{\lambda}

\newcommand{\om}{\omega}

\newcommand{\vp}{\varphi}

\newcommand{\La}{\Lambda}

\newcommand{\ul}{\underline}
\newcommand{\ol}{\overline}

\renewcommand{\d}{\partial}

\newcommand{\ddbar}{\sqrt{-1}\d\overline{\d}}

\newcommand{\ii}{\sqrt{-1}}

\newcommand{\RR}{\mathbb{R}}
\newcommand{\CC}{\mathbb{C}}

\newcommand{\PP}{\mathbb{P}}

\newcommand{\sK}{\mathcal{K}}

\newcommand{\sM}{\mathcal{M}}

\newcommand{\sO}{\mathcal{O}}

\newcommand{\fh}{\mathfrak{h}}
\newcommand{\Rc}{\mathrm{Ric}}

\newcommand{\Fut}{\mathrm{Fut}}

\makeatletter
\renewcommand*{\eqref}[1]{%
	\hyperref[{#1}]{\textup{\tagform@{\ref*{#1}}}}%
}
\makeatother
\title[Twisted and Coupled cscK metrics on ruled surfaces]{Twisted and Coupled constant scalar curvature K\"ahler metrics on minimal ruled surfaces}
\author{Ramesh Mete}
\address{Department of Mathematics, Indian Institute of Technology (IIT), Bombay, Powai, Mumbai - 400076, India.}
\email{rameshm@math.iitb.ac.in, rameshmete@alum.iisc.ac.in, ramesh2025m@gmail.com}

\thanks{Work supported in part by an Institute Post Doctoral Fellowship (IPDF) from the Indian Institute of Technology (IIT), Bombay.}

\begin{document}
	
\maketitle
	
\begin{abstract}
In this paper, we study the existence of twisted constant scalar curvature K\"ahler (cscK) metrics and non-existence of coupled cscK metrics on minimal ruled surfaces over a Riemann surface of genus $2$. Moreover, we give a bound for the Chen-Cheng invariant related to Chen's continuity path for cscK problem on these ruled surfaces.
\end{abstract}

%-------------------------------------
\section{Introduction}
	
A central problem in K\"ahler geometry for several decades is to find {\em canonical metrics}, and the important examples of canonical metrics are {\em constant scalar curvature K\"ahler} ({\em cscK}) metrics. A K\"ahler metric $\om$ on a K\"ahler manifold $M$ is called a {\em cscK metric} if its scalar curvature $R(\om)$ is constant on $M$. Given any K\"ahler class $[\om_0]$ on $M$, cscK metrics in $[\om_0]$ are exactly critical points of the {\em $K$-energy functional} $\sM_{\om_0}$, which was introduced by Mabuchi \cite{Mab86}. There are some obstructions for existence of cscK metrics in $[\om_0]$ and one of them is non-vanishing of the {\em Futaki invariant} \cite{Fut83, Fut83b}, which is an invariant of $[\om_0]$ and plays an important role in K\"ahler geometry (see \cite{Fut88-bk}). More precisely, if the K\"ahler class $[\om_0]$ on a compact K\"ahler manifold $M$ admits a cscK metric, then this Futaki invariant must be zero. But the converse is not true in general. Along with the cscK problem, both the {\em twisted version} and {\em coupled version} of it have been studied recently.

\vspace*{1mm}
Let $(M, \om_0)$ be a K\"ahler manifold of complex dimension $n$. Suppose $\chi$ is a smooth closed real $(1,1)$-form on $M$. Following Stoppa \cite{Sto09-tcscK}, a K\"ahler metric $\om \in [\om_0]$ is called a {\em $\chi$-twisted cscK metric} if it satisfies 
\begin{equation}\label{eq:twisted-cscK}
R(\omega) - \La_{\om}\chi = \ul{R} - \ul{\chi},
\end{equation}
where $R(\om)$, $\ul R$ are the scalar curvature of the metric $\om$ and its average, and the constant $\ul\chi$ is the average of the trace function $\La_\om\chi := \frac{n \chi \wedge \om^{n-1}}{\om^n}$, i.e. $$ \ul{R} := n \frac{2\pi c_1(M)\cdot [\om_0]^{n-1}}{[\om_0]^n},~~ \ul\chi := n \frac{[\chi]\cdot [\om_0]^{n-1}}{[\om_0]^n}.$$ 
Here, $c_1(M)$ is the first Chern class of the K\"ahler manifold $M$. One can see \cite{Sto09-tcscK, Der15-tcscK, Che18-cont-path-tcscK, Hashimoto19-tcscK, Zen19-tcscK, ChenCheng21-cscK-exis, JianShiSong19-cscK-min-model, SDyre22-cscK-min-model} and references therein for more details about the twisted cscK equation. In particular, when $\chi = 0$ the equation \eqref{eq:twisted-cscK} is nothing but the {\em cscK equation}
\begin{equation}\label{eq:cscK}
R(\om)=\ul{R}.
\end{equation}
Note that in order to get the equation \eqref{eq:twisted-cscK} the cscK equation \eqref{eq:cscK} has been twisted with the following important equation 
\begin{equation}\label{eq:J-equ}
\La_{\om} \chi = \ul\chi,
\end{equation}
which is called the {\em J-equation} (or {\em Donaldson's equation}) and was introduced in \cite{Don99-j-eq} (see also \cite{Chen00a-low-bd-Mab}), where it was assumed that $\chi$ is also a K\"ahler metric. Observe that if the K\"ahler metric $\om$ solves the J-equation \eqref{eq:J-equ} for the fixed K\"ahler metric $\chi>0$, then $\om$ is a $\chi$-twisted cscK metric if and only if it is a cscK metric. But there may exists a K\"ahler metric solving the twisted cscK equation \eqref{eq:twisted-cscK} which is neither a cscK metric nor solves the J-equation \eqref{eq:J-equ}. Recently, the J-equation has been studied extensively, for instance - see \cite{Chen00a-low-bd-Mab, Chen-parabolic-flow, Weinkove04-convergence-surface, Weinkove06-convergence-higher-dim, SW08-J-flow-conv, FL-Calabi ansatz-2013, FLSW14-J-flow-on-Kah-surf-boundary-case, LejSze15-J-flow-stab, ColSze17-J-flow-toric, gaoChen21-J-dHYM, DatPin2021-gCMA-proj-mfds, Song20-NM-crit-J-eq, SDyre20-j-func, DMS-min-slopes-cmpx-Hess-eq, To23-degenrate-tJ-flow} and references therein.

\vspace*{1mm}
In \cite{Datar-Pingali Coup-cscK}, Datar-Pingali introduced the {\em coupled cscK equation} generalizing the coupled version of K\"ahler-Einstein equation \cite{Hultgren-Witt Nystrom2018} (see also \cite{Hultgren2019,Delcroix-Hultgren2021}) and it is a coupling of an equation similar to \eqref{eq:twisted-cscK} with another equation involving Ricci curvatures. Let $M$ be a K\"ahler manifold of complex dimension $n$ and $(\om_i)_{i=0}^{N}$ be a $(N+1)$-tuple of K\"ahler metrics on $M$. Define $$ \omega_{\text{sum}} := \la\sum_{i=0}^N \om_i, $$ where $\la=1$ when $c_1(M)>0$, $\la=-1$ when $c_1(M)<0$ and $\la=0$ when $c_1(M)=0$. Note that $\la\om_{\text{sum}}$ is a K\"ahler metric on $M$ when $\la=\pm 1$. Following \cite{Datar-Pingali Coup-cscK}, we say $(\om_i)$ a {\em coupled cscK} metric if the following two conditions are satisfied
\begin{equation}\label{eq:all-Ric-curv-are-equal-intro}
\Rc(\om_0)=\cdots =\Rc(\om_{N}),
\end{equation}
and
\begin{equation}\label{eq:coupled-cscK-2nd-equ-intro}
R(\om_0) - \La_{\om_0}\om_{\text{sum}} = \ul{R} - \ul{\om_{\text{sum}}}.
\end{equation} 
Here, $\ul R$ is the average scalar curvature of the K\"ahler class $[\om_0]$ and $$\ul{\om_{\text{sum}}} = n \frac{[\om_{\text{sum}}]\cdot[\om_0]^{n-1}}{[\om_0]^n}.$$We observe that the equation \eqref{eq:all-Ric-curv-are-equal-intro} can also be re-written as
\begin{equation*}
\frac{1}{[\om_0]^n} \om_0^{n} = \cdots = \frac{1}{[\om_{N}]^n} \om_{N}^{n}.
\end{equation*}
In particular, for $N=1$ the pair of K\"ahler metrics $(\om_0, \om_1)$ is a coupled cscK metric if and only if $\Rc(\om_0) = \Rc(\om_1)$ and $\om_0$ is a $\la \om_1$-twisted cscK metric. Therefore, if $\om$ and $\chi$ are two K\"ahler metrics on a {\em Fano manifold} $M$ (i.e. $c_1(M)>0$ and so we take $\la=1$), then the pair $(\om, \chi)$ is a coupled cscK metric if and only if $\Rc(\om) = \Rc(\chi)$ and $\om$ is a $\chi$-twisted cscK metric.

\vspace*{1mm}
In this paper, we provide example of certain compact K\"ahler surfaces such that any pair of (normalized) K\"ahler classes don't admit a coupled cscK metric but there are some pairs of (normalized) K\"ahler classes which admit twisted cscK metrics.

\vspace*{1mm}
Let $X :=\PP(L\oplus \sO)$ be a {\em minimal ruled (complex) surface} (see $\S$ \ref{subsec:ruled-surf-over-genus-2-Riemm-surf}), where $L$ is a degree $-1$ holomorphic line bundle over a genus $2$ Riemann surface $(\Sigma, \om_\Sigma)$ equipped with a K\"ahler metric $\om_{\Sigma}$ whose scalar curvature $R(\om_{\Sigma})=-2$ and $\sO$ is the trivial line bundle over $\Sigma$. Further, we assume that there is a Hermitian metric $h$ on $L$ whose Chern curvature form $\ga(h)=-\om_\Sigma$. In particular, $c_1(L)<0$. In fact, $X$ is a $\CC\PP^1$ bundle over the Riemann surface $\Sigma$ and it is an example of {\em pseudo-Hirzebruch surfaces} mentioned in \cite{Ton98-extrem}.

\vspace*{1mm}
Let $\mathrm{C}$ be the Poincar\'{e} dual of a fibre of $X$ (i.e. $\mathrm{C}$ is a copy of the Riemann sphere $\mathbb{S}^2$ sitting inside $X$) and let $D_\infty$ be the {\em infinity divisor} of $X$, i.e. the image of the sub-bundle $L\oplus \{0\} \subset L\oplus\sO$ under the bundle projectivization map to $X$ (see $\S$ \ref{subsec:ruled-surf-over-genus-2-Riemm-surf} for more details). It is known \cite{Fuj92, LeBrun-Singer93, Ton98-extrem} that any K\"ahler class $\Omega$ on the ruled surface $X$ is of the form $\Omega = a_1 \mathrm{C} + a_2 D_\infty$ for some positive constants $a_1, a_2 > 0$. For any $m>0$ we consider the normalized K\"ahler class $\Omega_{m}:= 2\pi(\mathrm{C} + m D_{\infty})$ on the ruled surface $X$.

\vspace*{1mm}
Our first result is that there are certain pairs of normalized K\"ahler classes $(\Omega_a, \Omega_b)$ on the minimal ruled surface $X$ which admit K\"ahler metrics $\om\in\Omega_a$ and $\chi\in\Omega_b$ such that $\om$ is a $\chi$-twisted cscK metric.

\begin{thm}\label{thm:exis-of-twisted-cscK-on-min-ruled-surfaces}
Suppose $\Omega_a = 2\pi(\mathrm{C} + a D_{\infty})$ and $\Omega_b = 2\pi(\mathrm{C} + b D_{\infty})$ are two normalized K\"ahler classes on the ruled surface $X = \PP(L\oplus \sO)$ defined above with $b\geq \frac{a(5a+4)}{2(1+a)}$ and $a>0$. Then there exists a K\"ahler metric $\om\in\Omega_a$ whose momentum profile $\phi$ (see $\S$ \ref{subsec:momentum-profile}) is given by $$\phi(\tau)=\frac{a\tau-\tau^2}{a}, ~~~\tau\in[0,a]$$ such that $\om$ is a $\chi$-twisted cscK metric for some K\"ahler metric $\chi$ in the class $\Omega_b$.
\end{thm}

\begin{rem}
In \cite[Remark 1.5]{SDyre22-cscK-min-model}, Sj\"ostr\"om Dyrefelt showed that if the class $-c_1(X) + [\chi]$ is nef\footnote{A class $[\theta]\in H^{1,1}(M,\RR)$ is called {\em numerically effective} (in short, {\em nef}) if $[\theta] + \delta[\om] > 0$ for any $\del > 0$. Equivalently, $[\theta]$ is nef if it lies in the closure of the K\"ahler cone $\sK(M)$.} for any real $(1,1)$-cohomology class $[\chi]\in H^{1,1}(X,\RR)$ on any compact K\"ahler manifold $X$, then for any K\"ahler class $\al$ on $X$, there is a constant $\epsilon_{X,\al} > 0$ such that for all $0< \epsilon < \epsilon_{X,\al}$, there exists a unique $\chi$-twisted cscK metric in the K\"ahler class $\be := -c_1(X) + [\chi] + \epsilon \al$. In particular, for the ruled surface $X = \PP(L\oplus \sO)$ we have $c_1(X) = - 3 \mathrm{C} + 2 D_\infty$ (see $\S$ \ref{subsec:1st-Chern-class-average-scal-curv-computation}). It implies that $-c_1(X) + \Omega_b$ is a nef class if $b \geq 1/\pi$ for any normalized K\"ahler class $\Omega_b = 2\pi(\mathrm{C} + b D_{\infty})$. Therefore, for any K\"ahler class $\al$ on the ruled surface $X$, there is a constant $\epsilon_{X,\al} > 0$ such that the K\"ahler class $-c_1(X) + \Omega_b + \epsilon \al$ admits a $\chi$-twisted cscK metric, where $\epsilon\in(0, \epsilon_{X,\al})$ and $\chi$ is a K\"ahler metric in $\Omega_b$. But note that the K\"ahler class $-c_1(X) + \Omega_b + \epsilon \al$ is not in normalized form $\Omega_a$ for some $a>0$.
\end{rem}

Our next result is that although there are pairs of normalized K\"ahler classes on $X$ containing twisted cscK metrics but no such pair admits a coupled cscK metric.
	
\begin{thm}\label{thm:normalized-Kah-classes-not-admit-coupled-csck}
The pair of normalized K\"ahler classes $(\Omega_a = 2\pi(\mathrm{C} + a D_{\infty}), \Omega_b = 2\pi(\mathrm{C} + b D_{\infty}))$, where both $a,b>0$,  on the ruled surface $X = \PP(L\oplus \sO)$ does not admit any coupled cscK metric.
\end{thm}
	
In fact, using some scaling argument we show that any pair $(\al, \be)$ does not admit a coupled cscK metric, where $\al$ is {\em any} K\"ahler class and $\be$ is a {\em normalized} K\"ahler class on the minimal ruled surface $X$ (see Corollary \ref{cor:the-pair-of-any-Kah-and-normal-Kah-not-admit-coupled-cscK}). Further, even $\be$ is {\em not normalized} we show that $(\al, \be)$ does not admit a coupled cscK metric under some condition. For the later result we need to use the following solvability criteria for the J-equation \eqref{eq:J-equ}.
	
\begin{thm}\label{thm:solvability-of-J-equ-under-some-condition}
Suppose $\Omega_a = 2\pi(\mathrm{C} + a D_{\infty})$ and $\Omega_b = 2\pi(\mathrm{C} + b D_{\infty})$ are two normalized K\"ahler classes on the ruled surface $X = \PP(L\oplus \sO)$, where $a, b >0$. If $ b \geq \frac{a^2}{2(1+a)}$, then there exists K\"ahler metrics $\om \in \Omega_a$ and $\chi \in \Omega_b$ solving the J-equation \eqref{eq:J-equ}.
\end{thm}
	
As a corollary of Theorem \ref{thm:normalized-Kah-classes-not-admit-coupled-csck} and Theorem \ref{thm:solvability-of-J-equ-under-some-condition} we have the following result.
	
\begin{cor}\label{cor:the-pair-of-any-Kah-and-normal-Kah-not-admit-coupled-cscK}
Consider any two K\"ahler classes $\al:= a_1 \mathrm{C} + a_2 D_\infty$ and $\be := b_1 \mathrm{C} + b_2 D_\infty$, where $a_i, b_i >0$ for $i=1, 2$, on the ruled surface $X = \PP(L\oplus \sO)$. If $b_1 = 2\pi$, i.e. the class $\be$ is in the normalized form $\be = 2\pi (\mathrm{C} + \frac{b_2}{2\pi} D_\infty)$, then the pair $(\al, \be)$ does not admit any coupled cscK metric on $X$. Moreover, when $b_1 \neq 2\pi$ the pair $(\al, \be)$ does not admit a coupled cscK metric on $X$ if the following holds
\begin{equation}\label{eq:ineq-imply-non-exis-for-non-normalized-case}
\frac{b_2}{b_1} > \frac{a_{2}^2}{2 a_1 (a_1 + a_2)}.
\end{equation}
\end{cor}
	
\begin{rem}
Note that the condition \eqref{eq:ineq-imply-non-exis-for-non-normalized-case} is {\em not} necessary for non-existence of coupled cscK metrics on the ruled surface $X$. We expect that this condition can be removed, i.e. all pairs $(\al, \be)$ of K\"ahler classes on $X$ have no coupled cscK metrics, where $\be$ need not to be normalized.
\end{rem}
	
The Futaki invariant has been generalized in different settings (see \cite{Ban06-extended-Fut} for instance). In \cite{Datar-Pingali Coup-cscK}, Datar-Pingali introduced the {\em generalized (or coupled) Futaki invariant} for any $N$-tuple of K\"ahler classes $(\al_1, \cdots, \al_N)$ on a compact K\"ahler manifold $M$ (see $\S$ \ref{subsec:gen-Fut-inv-Datar-Pingali-coupled-cscK}) and they proved that this invariant must vanish identically if the $N$-tuple admits a coupled cscK metric $(\om_1,\cdots,\om_N)$, where $\om_i \in \al_i$ for $i\in\{1,\cdots,N\}$, i.e. solves the equations \eqref{eq:all-Ric-curv-are-equal-intro} and \eqref{eq:coupled-cscK-2nd-equ-intro}. We will compute this generalized Futaki invariant for the ruled surface $X$ mentioned above.

\vspace*{1mm}
In \cite{Che18-cont-path-tcscK}, Chen introduced the following continuity path to find a cscK metric in a given K\"ahler class $[\om_0]$ on an $n$-dimensional K\"ahler manifold $M$:
\begin{equation}\label{eq:Chen-continuity-path-tcscK}
t R(\om_{\vp_t}) -(1-t)\La_{\om_{\vp_t}} \chi = t \ul{R} - (1-t) \ul\chi,
\end{equation}
where $0\leq t \leq 1$, $\chi$ is a fixed K\"ahler form with corresponding average $\ul\chi = \frac{n [\chi]\cdot[\om_0]^{n-1}}{[\om_0]^n}$, $\ul{R}$ is the average scalar curvature for the K\"ahler class $[\om_0]$, and $\om_{\vp_t}:=\om_0 + \ddbar\vp_t$ is a path of K\"ahler metrics in $[\om_0]$. Note that at $t=1$ we get the cscK equation \eqref{eq:cscK}; at $t=0$ we get the J-equation \eqref{eq:J-equ}; and at $t=1/2$ we have the twisted cscK equation \eqref{eq:twisted-cscK}. Further, Chen \cite{Che18-cont-path-tcscK} introduced the following quantity related to the above continuity path (i.e. equation \eqref{eq:Chen-continuity-path-tcscK})
\begin{equation}
R([\om_0], \chi) := \sup \left\{ t_0 \in [0,1]:~ \eqref{eq:Chen-continuity-path-tcscK}~\text{can be solved for any}~ 0\leq t \leq t_0 \right\},
\end{equation}
and {\em conjectured} that it is an invariant of the K\"ahler class $[\chi]$. Later, Chen-Cheng \cite{ChenCheng21-cscK-exis} proved that $R([\om_0],\chi)$ is independent of the choice of the metric $\chi$, and hence this is an invariant ({\em which will be denoted by $R([\om_0], [\chi])$}) of the K\"ahler class $[\chi]$. By definition, we have $0 \leq R([\om_0], [\chi]) \leq 1$. Moreover, Chen-Cheng \cite{ChenCheng21-cscK-exis} proved the following important result:
	
\begin{thm}[{\cite[Theorem 1.7]{ChenCheng21-cscK-exis}}]
Let $\chi$ be a K\"ahler form. If the Mabuchi $K$-energy is bounded from below on a compact K\"ahler manifold $(M, [\om_0])$, then $R([\om_0], [\chi]) = 1$ if and only if one can solve the J-equation $\La_{\om_\vp} \chi = \ul\chi$.
\end{thm}
	
We have the following result about bound of the Chen-Cheng invariant on ruled surfaces. 
	
\begin{prop}
\label{prop:bound-of-Chen-Cheng-invariant-for-tcscK}
Suppose $\Omega_a = 2\pi(\mathrm{C} + a D_{\infty})$ and $\Omega_b = 2\pi(\mathrm{C} + b D_{\infty})$ are two normalized K\"ahler classes on the ruled surface $X = \PP(L\oplus \sO)$ defined above with $b\geq \frac{a^2}{2(1+a)}$ and $a>0$. Then we have
\begin{equation*}
\frac{2b(a+1) - a^2}{3a^2 + 4a + 2b(a+1)} \leq R (\Omega_a, \Omega_b) < 1.
\end{equation*}
In particular, we have $\frac{a+2}{5a + 6} \leq R(\Omega_a, \Omega_a) < 1$ for any normalized K\"ahler class $\Omega_a = 2\pi(\mathrm{C}+a D_\infty)$, where $a>0$, on the ruled surface $X$.
\end{prop}
	
\vspace*{1mm}
\textbf{Organization of the paper:} In $\S$ \ref{sec:preliminaries}, we first recall the definition of generalized Futaki invariant for the coupled cscK problem. Then, after providing a brief description of minimal ruled surfaces, we demonstrate how a K\"ahler metric $\om$ can be expressed in terms a real-valued function $\phi$, called momentum profile, defined on some interval using Calabi ansatz. We finish $\S$ \ref{sec:preliminaries} by recalling scalar curvature (and Ricci curvature) formula in terms of this momentum profile, and computing the average scalar curvature of a normalized K\"ahler class on ruled surfaces. In $\S$ \ref{sec:Non-existence of coupled cscK metrics on ruled surfaces}, we compute the generalized Futaki invariant for the pair $(\Omega_a, \Omega_b)$ of normalized K\"ahler classes on ruled surfaces and show that it is not identically zero, which will prove Theorem \ref{thm:normalized-Kah-classes-not-admit-coupled-csck}. Also in $\S$ \ref{sec:Non-existence of coupled cscK metrics on ruled surfaces}, we prove Theorem \ref{thm:solvability-of-J-equ-under-some-condition} using Calabi ansatz method, and then using Theorems \ref{thm:normalized-Kah-classes-not-admit-coupled-csck} and \ref{thm:solvability-of-J-equ-under-some-condition} we finish the proof of Corollary \ref{cor:the-pair-of-any-Kah-and-normal-Kah-not-admit-coupled-cscK}. Lastly, in $\S$ \ref{sec:Existence of twisted cscK metrics on ruled surfaces}, we prove Theorem \ref{thm:exis-of-twisted-cscK-on-min-ruled-surfaces} using Calabi ansatz, and we prove Proposition \ref{prop:bound-of-Chen-Cheng-invariant-for-tcscK} regarding the bound of Chen-Cheng invariant in $\S$ \ref{sec:Bound of the Chen-Cheng invariant on ruled surfaces}.

\vspace{0.5mm}
%%%%%%%%%%%%%%%%%%%%%%%%%%%%%%%%%%%%%%%%%%%%%
\section{Preliminaries}
\label{sec:preliminaries}

%--------------------------------------------
\subsection{(Generalized) Futaki invariant formula for any K\"ahler manifold}
\label{subsec:gen-Fut-inv-Datar-Pingali-coupled-cscK}
	
For any $N$-tuple $(\al_1,\cdots,\al_N)$ of K\"ahler classes on a compact K\"ahler manifold $M$ we define the following sub-space of holomorphic vector fields
\begin{equation*}
\fh(\al_1,\cdots,\al_N) = \left\{\xi \in H^0(M,T^{1,0}M)\Big|~ \begin{matrix*}[c]
\xi =\nabla^{1,0}_{\om_j} f_j ~\text{for some} \\ f_j\in C^\infty(M,\CC)~\forall~j=1,\cdots,N \end{matrix*}
\right\}
\end{equation*}
where $\om_j\in\al_j$, $j=1,\cdots,N$, are any K\"ahler metrics. Here, $\nabla_{\om}^{1,0} f := \om^{k\ol{l}}\frac{\d f}{\d \ol{z}^l} \frac{\d}{\d z^k}$ for any function $f\in C^\infty(M,\CC)$ and any K\"ahler metric $\om$ on $M$. The space $\fh(\al_1,\cdots,\al_N)$ is independent of the choice of the metrics $\om_i \in \al_i$. The tuple $(f_1,\cdots,f_N)$ is called a {\em $N$-tuple of holomorphy potentials} for $\xi$ with respect to $(\om_1,\cdots,\om_N)$. Then the {\em generalized Futaki invariant}, which was introduced in \cite{Datar-Pingali Coup-cscK}, is defined by
\begin{equation}
\begin{split}
\Fut (\al_1,\cdots,\al_N, \xi) &:= \frac{1}{\al_1^n} \int_{M} f_1 \left(R(\om_1) - \La_{\om_1}\om_{\text{sum}} - \ul{R}_{\al_1} + \ul{\om_{\text{sum}}}\right) \om_1^{n} \\ &\hspace*{1.8cm} + \sum_{j=1}^{N} \int_{M} f_{j} \left(\frac{\om_j^{n}}{\al_j^{n}} - \frac{\om_1^{n}}{\al_1^{n}}\right),
\end{split}
\end{equation}
where $\al_{j}^n := [\om_j]^n := \int_{M}\om_{j}^{n}$, $\om_{\text{sum}}:= \sum_{i=1}^{N} \om_i$, and the averages $\ul{R}_{\al_1}$, $\ul{\om_{\text{sum}}}$ are computed with respect to the K\"ahler class $\al_1$, i.e.
\begin{equation*}
\ul{R}_{\al_1} = n\frac{2\pi c_1(M) \cdot\al_1^{n-1}}{\al_1^n}, ~~~ \ul{\om_{\text{sum}}} = n\frac{(\sum_{i=1}^{N}\al_i)\cdot\al_1^{n-1}}{\al_1^n}.
\end{equation*}
It is known \cite{Datar-Pingali Coup-cscK} that the generalized Futaki invariant is independent of the choice of $\om_i \in \al_i$ (for each $i$) and the choice of holomorphy potential $(f_1,\cdots, f_N)$ for $\xi$ with respect to $(\om_1,\cdots, \om_N)$. In particular, for $N=1$ we have 
\begin{equation}
\Fut(\al,\xi) = \frac{1}{\al^n} \int_{M} f (R(\om) - \ul{R}) \om^n,
\end{equation}
where $\om\in\al$ is any K\"ahler metric and $f$ is a holomorphy potential for holomorphic vector field $\xi$ with respect to $\om$, which is the usual Futaki invariant introduced by Futaki \cite{Fut83, Fut83b} (see also \cite{Fut88-bk, Sze-bk}).

%-------------------------------------------
\subsection{Ruled surfaces over Riemann surfaces of genus 2}
\label{subsec:ruled-surf-over-genus-2-Riemm-surf}
	
Let $(\Sigma, \om_\Sigma)$ be a genus $2$ Riemann surface equipped with a K\"ahler metric $\om_{\Sigma}$ whose scalar curvature $R(\om_{\Sigma})=-2$. Since the Euler characteristic of $\Sigma$ is $\chi(\Sigma)=2(1-2)=-2$, by {\em the Gauss-Bonnet formula}, we have $\text{Area}(\Sigma, \om_\Sigma) :=\int_{\Sigma}\om_{\Sigma} = 2\pi.$ Let $(L,h)$ be a degree $-1$ holomorphic line bundle over $\Sigma$ equipped with a Hermitian metric $h$ whose Chern curvature form $\ga(h):=-\ddbar\log\det(h_{j\bar k})$ equals to $-\om_\Sigma$. In particular, $c_1(L)<0$. By the definition of degree\footnote{The degree of $L$ equals to $d$ means $\int_{\Sigma}c_1(L)=d$. Note that $c_1(L)=\frac{1}{2\pi}[\ga(h)]$.}, we have $\int_\Sigma c_1(L) = -1$ which implies $\int_\Sigma \ga(h) = -2\pi$ for any Hermitian metric $h$ on $L$. Then the K\"ahler surface $$X :=\PP(L\oplus \sO)$$ is called a {\em minimal ruled (complex) surface}. Note that $X$ is a $\PP^1$ bundle over $\Sigma$ and it is an example of ``pseudo-Hirzebruch surfaces" (cf. \cite{Ton98-extrem}).

\vspace*{1mm}
Let $\mathrm{C}$ be the Poincar\'{e} dual of a fibre of $X$ (i.e., $\mathrm{C}$ is a copy of the Riemann sphere $\mathbb{S}^2$ sitting inside $X$). Further, let $D_\infty$ (resp. $D_0$) be the image of the sub-bundle $L\oplus \{0\} \subset L\oplus\sO$ (resp. $\{0\}\oplus \sO \subset L\oplus\sO$) under the bundle projectivization map to $X$. $D_\infty$ (resp. $D_0$) is called the {\em infinity} (resp. {\em zero}) divisor of $X$. Both $D_\infty$ and $D_0$ are actually copies of $\Sigma$ sitting inside $X$ as its infinity and zero divisors respectively (and $\Sigma$ is identified with $D_0$ as a (complex) curve in $X$). The following intersection formulae hold (see \cite{BHPdeVen-bk, Sze-bk, Ton98-extrem}):
\begin{equation}\label{eq:intersection formluae for ruled surface with genus 2}
\begin{split}
& C^2 =0,~~  D_{\infty}^2 =1, ~~D_{0}^2 = D_{0}\cdot [\Sigma] = -1, \\ & C\cdot D_{\infty} = C\cdot D_{0} = C\cdot [\Sigma] = 1,~~ D_{\infty}\cdot D_{0} = D_{\infty}\cdot[\Sigma] = 0,\\ & c_1(L)\cdot [\Sigma] = -1,~~ [\om_\Sigma]\cdot [\Sigma] = 2\pi,
\end{split}
\end{equation}
where $[\Sigma] \in H^2(\Sigma, \RR)$ is the {\em fundamental class} of $\Sigma$ (and $[\Sigma]$ is identified with $D_0$ in $H^2(X, \RR)$). The last two formulas are follows from $\deg(L)=-1$ and $\text{Area}(\Sigma,\om_\Sigma)=2\pi$ respectively. Since $\Sigma$ is a Riemann surface, we have $H^2(\Sigma,\RR) = H^{1,1}(\Sigma,\RR)$.

\vspace*{1mm}
From the Leray-Hirsch theorem (see \cite{Hatcher-bk}), we have $$ H^2(X,\RR) = \RR \mathrm{C}\oplus \RR D_{\infty}.$$ The works of \cite{Fuj92, LeBrun-Singer93, Ton98-extrem} imply that a class $\be \in H^{1,1}(X,\RR)$ is K\"ahler if and only if $$\be^2 >0,~~ \be \cdot \mathrm{C} >0,~~ \be\cdot D_\infty >0,~~ \be \cdot D_0 >0.$$ In particular, the {\em K\"ahler cone} (i.e. the set of all K\"ahler classes) in $X$ is given by
$$ \sK(X) = \{ a\mathrm{C} + b D_{\infty} \in H^{1,1}(X,\RR):~ a > 0,~ b >0\}.$$

%----------------------------------------
\subsection{Calabi ansatz and momentum profile}
\label{subsec:momentum-profile}
	
Consider a (normalized) K\"ahler class $\Omega_a := 2\pi(\rm{C} + a D_\infty)$ (where $a>0$). We want to construct a K\"ahler form $\om$ in this K\"ahler class satisfying {\em Calabi ansatz/symmetry} (cf. \cite{Cal82}). On $X\setminus(D_\infty \cup D_0)$ we consider the logarithm of the fiber-wise norm $$s = \ln |(z,w)|_{h}^2 = \ln h(z) + \ln |w|^2$$ where $z$ being a local coordinate on $\Sigma$ and $w$ being a fibre coordinate on $L$, and $h$ denotes the hermitian metric on $L$ mentioned before. Write $p : L \to \Sigma$ for the projection map and consider the following $(1,1)$-form (cf. \cite{Cal82, HwangSinger02, Sze-bk}):
\begin{equation}\label{eq:Calabi-ansatz-express-for-omega}
\om = p^\ast\om_\Sigma + \ddbar u(s),
\end{equation}
where $u \in C^\infty(\RR, \RR)$. If $u$ is a strictly convex function such that the function $s \mapsto s+u(s)$ is strictly increasing, then $\om$ is a K\"ahler form on $X\setminus(D_\infty \cup D_0)$. Now, in order to extend it to whole $X$ and to belong in the K\"ahler class $\Omega_a$, the following asymptotic conditions must be satisfied:
\begin{equation}\label{eq:aymptotic-limit-for-derivative-of-u}
\lim_{s \to -\infty} u'(s) = 0,~~\lim_{s \to \infty} u'(s) = a.
\end{equation}
	
We change coordinates to $\tau := u'(s) \in [0,a]$ and let $\phi(\tau) := u''(s)$. The asymptotic conditions for $u'$ will imply (see \cite{Sze-bk, HwangSinger02})
\begin{equation}\label{eq:boundary conditions for momentum profile}
\phi(0)=\phi(a)=0,~~\phi'(0)=1,~~\phi'(a)=-1.
\end{equation}
Also note that $\phi$ is positive on $(0,a)$ since $u$ is a strictly convex function. In general, any smooth function $\phi : [0,a] \to \RR$ with $\phi >0$ on $(0,a)$ and satisfies the above boundary conditions \eqref{eq:boundary conditions for momentum profile} is called a {\em momentum profile}. So for every metric $\om$ satisfying Calabi symmetry there is a momentum profile.

%---------------------------------------------
\subsection{Ricci and scalar curvature formula in terms of momentum profile}
	
As in \cite{Sze-bk}, for a special choice of local trivialization of $L$ the K\"ahler metric $\om$ is given by
\begin{equation}\label{eq:local-expression-for-omega}
\om = (1+ u') p^\ast\om_\Sigma + u''  \frac{\sqrt{-1} dw \wedge d\bar{w}}{
|w|^2} = (1+\tau) p^\ast\om_\Sigma + \phi(\tau) \frac{\sqrt{-1} dw \wedge d\bar{w}}{|w|^2}.
\end{equation}
Then the volume form with respect to $\om$ is given by
\begin{equation}\label{eq:vol-form-express-ruled-surface}
\frac{\om^2}{2} = (1+\tau)\phi(\tau)p^\ast\om_\Sigma \wedge \frac{\ii dw \wedge d\bar{w}}{|w|^2} = (1+\tau) p^\ast\om_\Sigma \wedge d\tau \wedge d\theta.
\end{equation}
Here the later expression comes from the following observation $$ \phi(\tau)\frac{\ii dw \wedge d\bar{w}}{|w|^2} = d\tau \wedge d\theta$$ where $\theta$ is the argument of $w$ in polar coordinate. Indeed, writing $w=r e^{i\theta}$ we have $s= \ln |w|^2 = 2\ln r$ and $\frac{\ii dw \wedge d\bar{w}}{|w|^2} = \frac{2}{r}dr \wedge d\theta = ds \wedge d\theta$, and the observation follows from $d\tau = \phi(\tau)ds$. Now using \eqref{eq:vol-form-express-ruled-surface} the Ricci curvature of the metric $\om$ is given by
\begin{equation*}
\textnormal{Ric}(\om) = \left( -2 - \frac{[(1+\tau)\phi]'}{1+\tau} \right) p^\ast\om_\Sigma  -  \left( \frac{[(1+\tau)\phi]'}{1+\tau} \right)' \phi(\tau)\frac{\ii dw \wedge d\bar{w}}{|w|^2}.
\end{equation*}
Then using the formula $R(\om)\om^{n}= n \Rc(\om)\wedge\om^{n-1}$ (here $n=2$) the scalar curvature for the metric $\om$ is given by 
\begin{equation}\label{eq:scal-curv-formula-in-terms-of-momentum-profile}
R(\om) = -\frac{2}{1+\tau} - \frac{[(1+\tau)\phi]''}{1+\tau}.
\end{equation}

%------------------------------------------
\subsection{First Chern class and average scalar curvature}
\label{subsec:1st-Chern-class-average-scal-curv-computation}
	
Now we see that the first Chern class of $X$ is given by $$c_1(X) = -3\mathrm{C} + 2 D_\infty.$$ Indeed, writing $2\pi c_1(X) = R_1 \mathrm{C} + R_2 D_\infty$ for some $R_1, R_2 \in \RR$, and from \eqref{eq:intersection formluae for ruled surface with genus 2} we have $R_2 = 2\pi c_1(X)\cdot \mathrm{C}$ and $ R_1 + R_2 = 2\pi c_1(X)\cdot D_\infty$. From the Ricci formula above we compute
\begin{align*}
2\pi c_1(X)\cdot \mathrm{C} &= - \int_{\CC\setminus\{0\}} \left( \frac{[(1+\tau)\phi]'}{1+\tau} \right)' \phi(\tau)\frac{ i dw \wedge d\bar{w}}{|w|^2} \\ &= -\int_{0}^{2\pi}\int_{0}^{a} \left( \frac{(1+\tau)\phi' + \phi}{(1+\tau)} \right)' d\tau \wedge d\theta \\
&= 4\pi,
\end{align*}
and 
$$ 2\pi c_1(X)\cdot D_{\infty} = \left ( -2 -\frac{(1+a)\phi'(a)+\phi(a)}{1+a}\right) \int_{\Sigma} \om_{\Sigma} = -2\pi, $$
where we used the boundary conditions \eqref{eq:boundary conditions for momentum profile}. Therefore, $R_2 = 4\pi$ and $R_1 = -6\pi$.

\vspace*{1mm}
From the intersections formulae \eqref{eq:intersection formluae for ruled surface with genus 2} we compute
\begin{equation}\label{eq:values-of-intersection-prod-of-Chern-class-and-normalized-Kah-class}
\begin{split}
2\pi c_1(X)\cdot [\om] &= 2\pi (-3\mathrm{C} + 2 D_\infty) \cdot 2\pi (\mathrm{C} + a D_\infty) = 4\pi^2 (2-a), \\
[\om]^2 &= 4\pi^2 (\mathrm{C} + a D_\infty)^2 = 4\pi^2 (a^2 + 2a).
\end{split}
\end{equation}
Thus, the average scalar curvature for the class $[\om]=2\pi(\mathrm{C} + a D_\infty)$ is given by
\begin{align}\label{eq:average-scal-curv-value-for-normalized-Kah-class}
\ul{R} := 2 \frac{2\pi c_1(X)\cdot[\om]}{[\om]^2} = \frac{2(2-a)}{a^2 + 2a}.
\end{align}
	
\begin{rem}
One can also get the above intersection number values by integrating the corresponding local expressions of $\om$, $\Rc(\om)$ and $S(\om)$. For instance,
\begin{align*}
2\pi c_1(X)\cdot[\om] &=\int_{X} R(\om)\frac{\om^2}{2} \\ &= -\int_{\Sigma}\int_{0}^{2\pi}\int_{0}^{a} (2+[(1+\tau)\phi]'')p^\ast\om_\Sigma \wedge d\tau \wedge d\theta \\ &=  4\pi^2(2-a).
\end{align*}
\end{rem}

\vspace*{0.5mm}
%%%%%%%%%%%%%%%%%%%%%%%%%%%%%%%%%%%%
\section{Non-existence of coupled cscK metrics on ruled surfaces}
\label{sec:Non-existence of coupled cscK metrics on ruled surfaces}

%--------------------------------------------	
\subsection{Non-existence of usual cscK metrics}
\label{subsec:Fut-calc-for-one-Kah-class-ruled-surf-genus-2}
	
It is well-known that any K\"ahler class on the ruled surface $X$ does not admit a cscK metric (see \cite{Sze-bk} for instance). For the reader's convenience, we will give a short  proof of this fact by solving the ODE associated to the cscK equation \eqref{eq:cscK} and also by computing the Futaki invariant.

\vspace*{1mm}
We consider a normalized K\"ahler class $\Omega_{a} =2\pi\left(\mathrm{C} + a D_\infty\right)$, where $a>0$, on the ruled surface $X$. Suppose $\om\in\Omega_a$ is a {\em cscK metric} satisfying Calabi ansatz, i.e. $R(\om)=\ul{R}$ and there exists a strictly convex function $u\in C^\infty(\RR,\RR)$ satisfying \eqref{eq:aymptotic-limit-for-derivative-of-u} such that the K\"ahler metric $\om$ is given by \eqref{eq:Calabi-ansatz-express-for-omega}. In particular, we have a real variable $\tau=u'(s)\in(0,a)$ and the momentum profile $\phi:[0,a]\to\RR$, which is a smooth function with $\phi>0$ on $(0,a)$ and satisfies boundary condition \eqref{eq:boundary conditions for momentum profile}.

\vspace*{1mm}
Since the scalar curvature of $\om$ is given by \eqref{eq:scal-curv-formula-in-terms-of-momentum-profile}, the cscK equation $R(\om)=\ul R$ becomes the following second order ODE 
\begin{equation}\label{eq:2nd-order-ODE-scal-curv-equ-on-ruled-surf}
[(1+\tau)\phi]'' = -2 - \ul{R}(1+\tau).
\end{equation}
A general solution of this ODE is the following:
\begin{equation}
(1+\tau)\phi(\tau) = -\tau^2 - \ul{R}\frac{3\tau^2 + \tau^3}{6} + C_1 \tau + C_2
\end{equation}
for arbitrary constants $C_1, C_2$. Using $\phi(0)=0$ and $\phi'(0)=1$ (see \eqref{eq:boundary conditions for momentum profile}), we have $C_2 = 0$ and $C_1 = 1$. Since $\ul{R} = \frac{2(2-a)}{a^2 + 2a}$ (see \eqref{eq:average-scal-curv-value-for-normalized-Kah-class}), we get
\begin{equation}
\label{eq:uniq-sol-of-ODE-for-cscK-on-ruled-surf}
\phi(\tau) = \frac{1}{1+\tau} (\tau -\tau^2 + \frac{a-2}{3(a^2+2a)}(\tau^3 + 3 \tau^2)).
\end{equation}
But we compute that $\phi(a) = -\frac{2a^2}{3(a+2)} \neq 0$, which is a contradiction. Thus, the normalized K\"ahler class $\Omega_a$ on $X=\PP(L\oplus \sO)$ does not admit a cscK metric satisfying Calabi ansatz.

\vspace*{1mm}
We next compute the Futaki invariant. Let us consider the holomorphic vector field $\xi := \kappa w \frac{\d}{\d w}$, where $\kappa\in{\RR\setminus\{0\}}$ and $w$ is the fiber coordinate for the holomorphic line bundle $L$. Since for a function $f(\tau)$ of the variable $\tau$, we have $\nabla^{1,0}_{\om}f(\tau)=f'(\tau)w\frac{\d}{\d w}$, we see that $\xi = \nabla^{1,0}_{\om}f_{0}(\tau)$ for the holomorphy potential $f_0(\tau):=\kappa\tau$. So, $\xi \in \fh(\Omega_a)$ and the Futaki invariant is given by
\begin{equation}
\Fut (\al,\xi) = \frac{\kappa}{\Omega_{a}^2}\int_{X} \tau R(\om)\om^2 - \frac{\kappa \ul{R}}{\Omega_{a}^2}\int_{X}\tau\om^2.
\end{equation}
Recall $\Omega_{a}^2 = 4\pi^2(a^2+2a)$ and $\ul{R} = \frac{2(2-a)}{a^2 + 2a}$ (see \eqref{eq:values-of-intersection-prod-of-Chern-class-and-normalized-Kah-class} and \eqref{eq:average-scal-curv-value-for-normalized-Kah-class}). Now we compute
\begin{equation}\label{eq:integration-value-of-tau-single-Kah-metric}
\begin{split}
&\frac{1}{\Omega_{a}^2}\int_{X}\tau\om^2 =\frac{8\pi^2}{\Omega_{a}^2}\int_{0}^{a}\tau(\tau+1)d\tau= \frac{4\pi^2a^2(2a+3)}{3 \cdot \Omega_{a}^2} = \frac{(2a+3)a}{3(a+2)}, \\
&\frac{1}{\Omega_{a}^2}\int_{X}\tau R(\om)\om^2 = \frac{-8\pi^2}{\Omega_{a}^2}  \int_{0}^{a}\tau(2 + [(1+\tau)\phi]'')d\tau = \frac{8\pi^2 a}{\Omega_{a}^2} 
= \frac{2}{a+2}.
\end{split}
\end{equation}
Combining these values yield
\begin{equation}
\label{eq:Futaki value on ruled surface-single metric}
\Fut(\al,\xi) = \kappa(\frac{2}{a+2} - \frac{2(2a+3)(2-a)}{3(a+2)^2}) = \frac{4a(a+1)\kappa}{3(a+2)^2} \neq 0.
\end{equation}
Hence, there is no cscK metric in the normalized K\"ahler class $\Omega_{a} = 2\pi(\mathrm{C}+a D_\infty)$.

\vspace*{1mm}
Now, suppose $\Omega:= a_1\mathrm{C} + a_2 D_\infty$, where $a_1, a_2 >0$, is any K\"ahler class on $X$. Assume $\theta \in \Omega$ is a cscK metric. Then $\frac{2\pi}{a_1}\theta$ is a cscK metric in the normalized K\"ahler class $\Omega_{\frac{a_2}{a_1}} = 2\pi(\mathrm{C} + \frac{a_2}{a_1} D_\infty)$, which is a contradiction. Therefore, any K\"ahler class on the ruled surface $X$ does not admit a cscK metric.

%------------------------------------------
\subsection{Non-existence of coupled cscK metrics by computing Futaki invariant}
\label{sec:Fut-calc-two-Kah-classes}
	
We consider K\"ahler metrics $\om\in\Omega_a = 2\pi(\mathrm{C} + a D_\infty)$ and $\chi\in\Omega_b = 2\pi(\mathrm{C} + b D_\infty)$ both satisfying Calabi ansatz. Then $\om$ is locally given by \eqref{eq:local-expression-for-omega} in terms of the momentum profile $\phi$ satisfying \eqref{eq:boundary conditions for momentum profile}. Similarly, since $\chi$ also satisfies Calabi ansatz, we have
$$ \chi = p^\ast\om_\Sigma + \ddbar v(s),$$ 
for some strictly convex function $v\in C^\infty(\RR,\RR)$ such that $s \mapsto s+v(s)$ is strictly increasing function. The following asymptotic conditions for the function $v$ hold:
$$\lim_{s \to -\infty} v'(s) = 0,~\lim_{s \to \infty} v'(s) = b. $$
Further, for a special choice of local trivialization of the holomorphic line bundle $L$, the metric $\chi$ is locally written as (cf. \cite{Sze-bk})
\begin{equation}
\label{eq:local-express-chi}
\chi = (1+ v') p^\ast\om_\Sigma + v'' \frac{\sqrt{-1} dw \wedge d\bar{w}}{|w|^2}.
\end{equation}
We fix the variable $\tau :=u'(s)\in(0,a)$ corresponding to the metric $\om$ (see $\S$ \ref{subsec:momentum-profile}). Let $\psi : [0,a]\to[0,b]$ be the smooth function defined by  $\psi(\tau) :=v'(s)$ and satisfying the boundary condition
\begin{equation}\label{eq:boundary conditions for psi}
\psi(0)=0,~\psi(a)=b.
\end{equation}
Note that the function $\psi$ is bijective. Differentiating $\psi$ with respect to $s$ we have  $$\psi'(\tau)\phi(\tau) = v''(s). $$ So the function $[0,b]\ni\sigma \mapsto \psi'(\tau)\phi(\tau)$, where $\psi(\tau)=\sigma$, is the momentum profile for the K\"ahler metric $\chi$. Hence, from \eqref{eq:local-express-chi} we obtain the following local expression for $\chi$ in terms of momentum profile
\begin{equation}\label{eq:local-expression-for-chi-momentum-profile}
\chi =  (1+\psi(\tau)) p^\ast\om_\Sigma + \psi'(\tau)\phi(\tau) \frac{\sqrt{-1} dw \wedge d\bar{w}}{|w|^2}.
\end{equation}

\vspace*{1mm}
Using the intersection number formula \eqref{eq:intersection formluae for ruled surface with genus 2} we get
\begin{equation*}
[\chi]\cdot[\om] = 2\pi (\mathrm{C} + b D_\infty)\cdot 2\pi (\mathrm{C} + a D_\infty) = 4\pi^2 (b+a+ab).
\end{equation*}
Note that one can obtain the same value by directly integrating the local expression \eqref{eq:local-express-chi} of $\chi$. Recall $[\om]^2=\Omega_{a}^2 = 4\pi^2 (a^2 + 2a)$ and so we get
\begin{align}\label{eq:average-of-trace-of-chi-wrt-omega}
\ul{\chi} = 2 \frac{[\chi]\cdot[\om]}{[\om]^2} = 2 \frac{b+a+ab}{a^2+2a},
\end{align}
and recall one again that $\ul R = \frac{2(2-a)}{a^2 +2a}$ (see \eqref{eq:average-scal-curv-value-for-normalized-Kah-class}). We  will denote the difference $\ul R -\ul\chi$ by $\hat{S}_\chi$ (or simply $\hat S$). Then, 
\begin{equation}\label{eq:difference-of-average-scal-curv-chi-average}
\hat{S}_{\chi} := \hat{S} := \ul{R} - \ul{\chi} = \frac{2(2-a)}{a^2 +2a} - \frac{2(a+b+ab)}{a^2 +2a} =  \frac{2(2-2a-b-ab)}{a^2 +2a} .
\end{equation}

\vspace*{1mm}
Once again we consider the same the holomorphic vector field $\xi = \kappa w \frac{\d}{\d w}$, where $\kappa\in{\RR\setminus\{0\}}$, on the ruled surface $X$. Taking $f_0(\tau):= \kappa\tau, ~ f_1(\tau) := \kappa\psi(\tau)$ we see that  $\xi =\nabla^{1,0}_{\om} f_0=\nabla^{1,0}_{\chi} f_1$, i.e.
$\xi\in \fh(\Omega_a,\Omega_b)$ (note that $\nabla^{1,0}_{\chi}f_1(\tau) = \kappa \nabla^{1,0}_{\chi} \psi(\tau) = \kappa w\frac{\d}{d w} =\xi$ since $\psi(\tau)=v'(s)$ plays the same role for $\chi$ as $\tau$ plays for $\om$). Next we compute the generalized Futaki invariant at $\xi\in \fh(\Omega_a,\Omega_b)$
\begin{equation*}
\Fut(\Omega_a,\Omega_b,\xi) = \frac{\kappa}{\Omega_{a}^2}\int_{X}\tau \left(R(\om) -\La_{\om}\chi -\hat{S}\right)\om^{2} + \kappa\int_{X}\psi(\tau)\left(\frac{\chi^{2}}{\Omega_{b}^2} - \frac{\om^{2}}{\Omega_{a}^2}\right),
\end{equation*}
where $\hat{S}$ is the constant given in \eqref{eq:difference-of-average-scal-curv-chi-average}. From \eqref{eq:integration-value-of-tau-single-Kah-metric} we have the following
\begin{equation*}
\frac{1}{\Omega_{a}^2}\int_{X}\tau\om^2 = \frac{(2a+3)a}{3(a+2)}, ~~~ \frac{1}{\Omega_{a}^2}\int_{X}\tau R(\om)\om^2 = \frac{2}{a+2}.
\end{equation*}
Since $\psi(\tau)$ plays the role of $\tau$ for the metric $\chi$, similarly we get
\begin{equation*}
\frac{1}{\Omega_{b}^2}\int_{X}\psi(\tau)\chi^2= \frac{(2b+3)b}{3(b+2)}.
\end{equation*}
Next we compute
\begin{equation*}
\begin{split}
\frac{1}{[\om]^2}\int_{X}\tau(\La_{\om}\chi) \om^2 &= \frac{1}{[\om]^2}\int_{X} \tau ( \psi' + \frac{1+\psi}{1+\tau})\om^2 \\
&= \frac{8\pi^2}{[\om]^2} \int_{0}^{a}\tau\Big(1+\psi+\psi'(1+\tau)\Big)d\tau \\
& = \frac{a + 2b(1+a)}{a+2} -\frac{1}{[\om]^2}\int_{X}\psi(\tau)\om^2.
\end{split}
\end{equation*}
Putting these values and simplifying we obtain
\begin{align*}
\frac{1}{\kappa}\Fut(\Omega_a,\Omega_b,\xi) 
&=  \frac{2}{a+2} - \frac{\hat{S}(2a+3)a}{3(a+2)} - \frac{a + 2b(1+a)}{a+2} + \frac{(2b+3)b}{3(b+2)} \\
&= \frac{2 (5a^2 + b^2 + 2a^2b + 4a)}{3(a + 2)^{2}(b + 2)} > 0.
\end{align*}
Hence the pair $(\Omega_a,\Omega_b)$ of normalized K\"ahler classes on the ruled surface $X$ does not admit any coupled cscK metric. This proves Theorem \ref{thm:normalized-Kah-classes-not-admit-coupled-csck}.

\vspace*{1mm}
In general, suppose $\al = a_1\mathrm{C} + a_2 D_\infty$ and $\be = b_1\mathrm{C} + b_2 D_\infty$, where $a_1, a_2, b_1, b_2 > 0$, are two K\"ahler classes (not necessarily normalized) on the ruled surface $X$. Suppose $(\om,\chi)$ is a coupled cscK metric with $\om\in\al$ and $\chi\in \be$. Consider the following scaling:
$$ \al' :=2\pi(\mathrm{C} + \frac{a_2}{a_1} D_\infty),~~\be' := 2\pi(\mathrm{C} + \frac{b_2}{b_1} D_\infty),~~\om' := \frac{2\pi}{a_1}\om,~~ \chi' := \frac{2\pi}{b_1}\chi.$$ 
Then both $\al'$ and $\be'$ are normalized K\"ahler classes on $X$, and $\om'$ (resp. $\chi'$) is a K\"ahler metric in $\al'$ (resp. $\be'$). Since the Ricci curvature is scale invariant, we have $$\Rc(\om')=\Rc(\om)=\Rc(\chi)=\Rc(\chi'),$$ where the second equality holds because we assumed that $(\om,\chi)$ is a coupled cscK metric. It follows that $R(\om')=\frac{a_1}{2\pi}R(\om)$ and
$$\La_{\om'}\chi' = n \frac{\chi'\wedge (\om')^{n-1}}{(\om')^n} =  \frac{a_1}{b_1} \cdot n \frac{\chi\wedge \om^{n-1}}{\om^n} = \frac{a_1}{b_1}\La_{\om}\chi.$$
It implies that the corresponding averages are related by $$\ul{R}_{[\om']}=\frac{a_1}{2\pi}\ul{R}_{[\om]}, ~~ \ul{\chi'}_{[\om']} = \frac{a_1}{b_1}\ul{\chi}_{[\om]}. $$
Since $(\om,\chi)$ is a coupled cscK metric, we have $$R(\om)-\La_{\om}\chi = \ul{R}_{[\om]} -\ul{\chi}_{[\om]}.$$ 
Then
\begin{equation}\label{eq:twsited-cscK-after-scaling}
\begin{split}
R(\om') - \La_{\om'}\chi' - \ul{R}_{[\om']} + \ul{\chi'}_{[\om']} &= \frac{a_1}{2\pi}\left( R(\om)-\ul{R}_{[\om]} \right) - \La_{\om'}\chi' + \ul{\chi'}_{[\om']} \\
&= \frac{a_1}{2\pi}\left( \La_{\om}{\chi}-\ul{\chi}_{[\om]} \right) - \La_{\om'}\chi' + \ul{\chi'}_{[\om']} \\
&= \frac{a_1}{2\pi} \frac{b_1}{a_1} \left( \La_{\om'}\chi' - \ul{\chi'}_{[\om']} \right) - \La_{\om'}\chi' + \ul{\chi'}_{[\om']} \\
&= \left(\frac{b_1}{2\pi} -1 \right) \left( \La_{\om'}\chi' - \ul{\chi'}_{[\om']} \right).
\end{split}
\end{equation}
Therefore, if $b_1 = 2\pi$, i.e. $\be$ is a normalized K\"ahler class, then $(\om',\chi')$ is a coupled cscK metric for the pair $(\al', \be')$ of normalized K\"ahler classes on $X$, which is a contradiction by Theorem \ref{thm:normalized-Kah-classes-not-admit-coupled-csck}. Hence, the pair $(\al := a_1 \mathrm{C} + a_2 D_\infty, \be := 2\pi \mathrm{C} + b_2 D_\infty)$ of K\"ahler classes does not admit any coupled cscK metric, which proves the first part of Corollary \ref{cor:the-pair-of-any-Kah-and-normal-Kah-not-admit-coupled-cscK}.

\vspace*{1mm}
Next, assuming Theorem \ref{thm:solvability-of-J-equ-under-some-condition} we shall prove the second part of Corollary \ref{cor:the-pair-of-any-Kah-and-normal-Kah-not-admit-coupled-cscK} (i.e. non-existence of coupled cscK metrics when $b_1 \neq 2\pi$). Since $\al'= 2\pi(\mathrm{C} + \frac{a_2}{a_1} D_\infty)$ and $\be' = 2\pi(\mathrm{C} + \frac{b_2}{b_1} D_\infty)$ are normalized K\"ahler classes on $X$, by applying Theorem \ref{thm:solvability-of-J-equ-under-some-condition} we see the right hand side of \eqref{eq:twsited-cscK-after-scaling} is zero if 
\[ \frac{b_2}{b_1} > \frac{(\frac{a_2}{a_1})^2}{2 (1 + \frac{a_2}{a_1})},\] 
i.e. the condition \eqref{eq:ineq-imply-non-exis-for-non-normalized-case} holds, and hence $(\om', \chi')$ is a coupled cscK metric, which is a contradiction by Theorem \ref{thm:normalized-Kah-classes-not-admit-coupled-csck}. This completes the proof of the second part of Corollary \ref{cor:the-pair-of-any-Kah-and-normal-Kah-not-admit-coupled-cscK}.

\vspace*{1mm}
One can also verify that the left hand side of \eqref{eq:twsited-cscK-after-scaling} can be written as
\begin{align*}
R(\om') - \La_{\om'}\chi' =  \ul{R}_{[\om']} -\ul{\chi'}_{[\om']} + (\frac{a_1}{2\pi} -\frac{a_1}{b_1})(R(\om)-\ul{R}_{[\om]}).
\end{align*}
Since the class $\al=a_1 \mathrm{C}+a_2 D_\infty$ does not admit any cscK metric, the pair $(\om',\chi')$ is not a coupled cscK metric if $b_1\neq 2\pi$. So, we don't get any contradiction here. If we can show that the corresponding generalized Futaki invariant is not zero identically on $\fh(\al,\be)$, then the pair $(\al,\be)$ will not admit any coupled cscK metric. We expect this to be true.

\vspace*{1mm}
We now complete the proof of Theorem \ref{thm:solvability-of-J-equ-under-some-condition}.
	
\begin{proof}[Proof of Theorem \ref{thm:solvability-of-J-equ-under-some-condition}]
Consider the normalized K\"ahler classes $\Omega_{a} = 2\pi(\mathrm{C} + a D_\infty)$ and $\Omega_{b} = 2\pi(\mathrm{C} + b D_\infty)$, where $a, b > 0$. Let $\om\in\Omega_a$ and $\chi\in\Omega_b$ be K\"ahler metrics satisfying Calabi ansatz (see $\S$ \ref{subsec:momentum-profile}). In particular, we have the variable $\tau=u'(s)\in(0,a)$ and $\om$ is locally given by \eqref{eq:local-expression-for-omega} in terms of momentum profile $\phi$ satisfying \eqref{eq:boundary conditions for momentum profile}. Moreover, the metric $\chi$ is locally given in \eqref{eq:local-expression-for-chi-momentum-profile} in terms of the function $\psi : [0, a] \longrightarrow [0, b]$ satisfying the boundary condition \eqref{eq:boundary conditions for psi}. We then see that the eigenvalues of $\chi$ with respect to $\om$, i.e. of $\om^{-1}\chi$, are $\frac{1+\psi(\tau)}{1+\tau}$ and $\psi'(\tau)$, and so the trace $\La_{\om}\chi$ is given by 
\begin{equation}\label{eq:trace-formula-in-trems-of-psi}
\La_{\om} \chi = \frac{1+\psi(\tau)}{1+\tau} + \psi'(\tau).
\end{equation}
Therefore, the J-equation \eqref{eq:J-equ} is equivalent to the following ODE
\begin{equation}
\psi'(\tau) + \frac{1 + \psi(\tau)}{1 + \tau} = \ul\chi.
\end{equation}
A general solution of this ODE is given by
\begin{equation*}
(1+\tau) \psi(\tau) = \ul\chi (\tau + \frac{\tau^2}{2}) - \tau + C_1,
\end{equation*}
for a constant $C_1$. Using the boundary condition $\psi(0)=0$ we get that $C_1 = 0$. Now using the value $\ul{\chi} = 2 \frac{[\chi]\cdot[\om]}{[\om]^2} =  \frac{2(b+a+ab)}{a^2+2a}$ (see \eqref{eq:average-of-trace-of-chi-wrt-omega}) and a simple algebraic computation, we see that the unique solution is the following
\begin{equation}\label{eq:uniq-sol-for-J-equ-ruled-surf}
\psi(\tau) = \frac{a+b+ab}{a^2+2a}\tau + \frac{(b-a)(1+a)}{a^2+2a}\frac{\tau}{1+\tau}.
\end{equation}
One can also easily check that the other boundary condition $\psi(a)=b$ is satisfied. Differentiating $\psi$ once we get
$$ \psi'(\tau) = \frac{a+b+ab}{a^2+2a} + \frac{(b-a)(1+a)}{a^2+2a}\frac{1}{(1+\tau)^2}.$$
We see that $\psi$ is strictly increasing when $b\geq a$. On the other hand, if $b<a$, taking one more derivative we see that $\psi$ is strictly convex. At $\tau=0$ we have $\psi'(0)= \frac{2b(1+a)-a^2}{a^2+2a}$ which is non-negative if and only if $b\geq \frac{a^2}{2(1+a)}$. Therefore, $\psi$ is strictly positive on $(0,a)$ when $b\geq \frac{a^2}{2(1+a)} >0,$ and hence the J-equation \eqref{eq:J-equ} is solvable under the condition \eqref{eq:ineq-imply-non-exis-for-non-normalized-case}.
\end{proof}

\vspace*{0.5mm}
%%%%%%%%%%%%%%%%%%%%%%%%%%%%%%%%%%%%%%
\section{Existence of twisted cscK metrics on ruled surfaces}
\label{sec:Existence of twisted cscK metrics on ruled surfaces}

%-----------------------------------------
\subsection{Reduction of the twisted cscK equation on ruled surfaces}
	
Let $\om \in \Omega_{a} = 2\pi(\mathrm{C} + a D_\infty)$ and $\chi \in \Omega_{b} = 2\pi(\mathrm{C} + b D_\infty)$ be K\"ahler metrics satisfying Calabi ansatz, where $a,b>0$ (see $\S$ \ref{subsec:momentum-profile}). In particular, we have the variable $\tau=u'(s)\in(0,a)$ and $\om$ is locally given by \eqref{eq:local-expression-for-omega} in terms of momentum profile $\phi$ satisfying \eqref{eq:boundary conditions for momentum profile}. Moreover, the metric $\chi$ is locally given in \eqref{eq:local-expression-for-chi-momentum-profile} in terms of the function $\psi : [0, a] \longrightarrow [0, b]$ satisfying the boundary condition \eqref{eq:boundary conditions for psi}. Recall the twisted cscK equation \eqref{eq:twisted-cscK}
\begin{equation}
R(\om) - \La_\om \chi = \ul{R} -\ul\chi.
\end{equation}
Also, the value of the difference $\ul{R} -\ul\chi$ (which we denoted by $\hat{S}_{\chi}$ or simply $\hat{S}$) is given by (see \eqref{eq:difference-of-average-scal-curv-chi-average})
$$ \hat{S}_\chi = \ul{R} - \ul{\chi} =  \frac{2(2-2a-b-ab)}{a^2 +2a}. $$
Now using the formula \eqref{eq:scal-curv-formula-in-terms-of-momentum-profile} for the scalar curvature $R(\om)$ and the above formula \eqref{eq:trace-formula-in-trems-of-psi}, the twisted cscK equation \eqref{eq:twisted-cscK} becomes the following ordinary differential equation with boundary condition $\psi(0)=0$ and $\psi(a) = b$:
\begin{equation*}
\psi'(\tau) + \frac{\psi(\tau)}{1+\tau} = -\frac{3 + [(1+\tau)\phi]''}{1+\tau} +  \frac{2(2a + b + ab -2)}{a^2 +2a}. 
\end{equation*}
A general solution of this ODE is given by
$$ (1+\tau)\psi(\tau) = -(3\tau +(1+\tau)\phi' +\phi) + \frac{(2a + b + ab -2)}{a^2 +2a}(\tau^2 + 2\tau) + C,$$ where $C$ is a constant. Putting $\tau=0$ we get $C=\phi'(0)=1$. Therefore, if $\om$ is a $\chi$-twisted cscK metric, then the functions $\psi$, $\phi$ must satisfy the following 
\begin{equation}\label{eq:unique solution of the ODE BVP}
(1+\tau)\psi(\tau) = -\Big(3\tau +(1+\tau)\phi' +\phi\Big) + \frac{(2a + b + ab -2)}{a^2 +2a}(\tau^2 + 2\tau) + 1.
\end{equation}
One can check that the other boundary condition $\psi(a) = b$ is also satisfied.

%-----------------------------------------
\subsection{Existence of twisted cscK metrics for a choice of momentum profile}
	
Our aim is to find smooth functions $\phi:[0,a]\to \RR_{\geq 0}$ and $\psi:[0,a]\to[0,b]$ such that both $\phi,\psi >0$ on $(0,a)$ with the boundary conditions \eqref{eq:boundary conditions for momentum profile} (for $\phi$) and \eqref{eq:boundary conditions for psi} (for $\psi$) and the pair $(\phi,\psi)$ solves \eqref{eq:unique solution of the ODE BVP}. Let us consider 
\begin{equation}\label{eq:specific-momentum-profile-for-omega}
\phi(\tau):= \frac{a\tau -\tau^2}{a}
\end{equation}
where $\tau\in[0,a]$. We easily see that $\phi$ is positive on $(0,a)$ and satisfies \eqref{eq:boundary conditions for momentum profile}, i.e. $\phi$ is a momentum profile and it will give us the K\"ahler metric $\om$. Now, for this specific $\phi$, we have
\begin{equation}\label{eq:value-of-1-plus-tau-times-phi-prime-for-special-momentum-profile}
(1+\tau)\phi'(\tau) + \phi(\tau) = \frac{(1+\tau)(a-2\tau)}{a} + \frac{a\tau - \tau^2}{a}
= 1 + \frac{2(a - 1) \tau}{a} - \frac{3 \tau^2}{a}.
\end{equation}
Then from \eqref{eq:unique solution of the ODE BVP} and \eqref{eq:value-of-1-plus-tau-times-phi-prime-for-special-momentum-profile} we get
\begin{align*}
(1+\tau)\psi(\tau) &=  \frac{3}{a}\tau^2 + \frac{2-5a}{a}\tau -1 + \frac{(2a + b + ab -2)}{a^2 +2a}(\tau^2 + 2\tau) + 1 \\
&= \frac{ ab + 5a + b + 4}{a^2 + 2a} \tau^2 + \frac{-5a^2 - 4a + 2b + 2ab}{a^2+2a}\tau,
\end{align*}
and dividing both sides by $(1+\tau)$ we obtain
\begin{equation*}
\psi(\tau) = \frac{ ab + 5a + b + 4}{a^2 + 2a}\frac{\tau^2}{1+\tau} + \frac{-5a^2 - 4a +2b + 2ab}{a^2+2a}\frac{\tau}{1+\tau}.
\end{equation*}
Then differentiating $\psi$ with respect to $\tau$ we obtain
\begin{align*}
\psi'(\tau) &= \frac{ (ab + 5a + b + 4)}{a^2 + 2a}\frac{\tau^2+2\tau}{(1+\tau)^2} + \frac{-5a^2 - 4a +2b + 2ab}{a^2+2a}\frac{1}{(1+\tau)^2} \\
&= \frac{ (ab + 5a + b + 4)}{a^2 + 2a} + \frac{(a+1)(-5a-4+b)}{a^2+2a}\frac{1}{(1+\tau)^2}.
\end{align*}
It implies $\psi$ is strictly increasing when $b\geq(5a+4)$. While if $0< b<(5a+4)$, taking one more derivative we see that $\psi$ is strictly convex. In particular, $\psi'(\tau) > \psi'(0)$ for all $\tau\in(0,a)$, and putting $\tau=0$ we have $$\psi'(0)= \frac{-5a^2 - 4a +2b + 2ab}{a^2+2a},$$ which is non-negative if and only if $b\geq \frac{a(5a+4)}{2(1+a)}$. Thus, for the choice \eqref{eq:specific-momentum-profile-for-omega} of $\phi$ and $b\geq \frac{a(5a+4)}{2(1+a)}$ we have $\psi(\tau)\in(0,b)$ for all $\tau\in(0,a)$. One can also verify that $\psi$ satisfies the boundary conditions \eqref{eq:boundary conditions for psi}. Therefore, $\psi$ will give us a K\"ahler metric $\chi$, and the K\"ahler metric $\om$ whose momentum profile is given by \eqref{eq:specific-momentum-profile-for-omega} is a $\chi$-twisted cscK metric on the ruled surface $X=\PP(L\oplus \sO)$. This completes the proof of Theorem \ref{thm:exis-of-twisted-cscK-on-min-ruled-surfaces}.

\vspace*{0.5mm}
%%%%%%%%%%%%%%%%%%%%%%%%%%%%%%%%%%%%%%%%%%%
\section{Bound of the Chen-Cheng invariant $R([\om], [\chi])$ on ruled surfaces}
\label{sec:Bound of the Chen-Cheng invariant on ruled surfaces}
	
In this section, we will give a lower bound for the invariant $R(\Omega_a, \Omega_b)$ corresponding to the normalized K\"ahler classes $\Omega_a = 2\pi (\mathrm{C}+ a D_\infty)$ and $\Omega_b = 2\pi (\mathrm{C}+ b D_\infty)$ on the ruled surface $X = \PP(L\oplus \sO)$ and prove Proposition \ref{prop:bound-of-Chen-Cheng-invariant-for-tcscK}. For notational simplicity, we denote $\hat{S}_t := t\ul{R}-(1-t)\ul\chi$ for any $t\in[0,1]$. Since $\ul{R} = \frac{2(2-a)}{a^2 + 2a}$ and $\ul{\chi} = \frac{2(b+a+ab)}{a^2+2a}$ (see \eqref{eq:average-scal-curv-value-for-normalized-Kah-class} and \eqref{eq:average-of-trace-of-chi-wrt-omega}), we compute that
\begin{equation}\label{value-of-hat S t-on-ruled-surf}
\hat{S}_t = \frac{2t(2-a)}{a^2+2a} -\frac{2(1-t)(b+a+ab)}{a^2+2a} = \frac{2\Big( t(2 + b + ab) - a -b -ab\Big)}{a^2 + 2a}.
\end{equation}

\vspace*{1mm}
Fix $t\in[0,1]$ and suppose $\chi$ and $\om_{\vp_t}$ satisfy Calabi ansatz, i.e. $$ \chi = p^\ast\om_\Sigma + \ddbar v(s),~~ \om_{\vp_t} = p^\ast\om_\Sigma + \ddbar u(s,t),$$ 
for some real-valued function $v$ (see $\S$ \ref{sec:Fut-calc-two-Kah-classes}) and some real-valued function $u(\cdot, t)$ (see $\S$ \ref{subsec:momentum-profile}). In particular, we have the variable $\tau:= u'$ and the momentum profile $\phi(\tau):= u''$ for the K\"ahler form $\om_{\vp_t}$. Note that $\tau$ depends on the parameter $t$. As before, we consider the function $\psi_{t} : [0,a] \to [0,b]$ defined by $ \psi_{t} (u'(s,t)) = v'(s)$. Varying $t\in[0,1]$ we can also treat $\psi$ as the function from $[0,a]\times[0,1]$ to $[0,b]$. Using \eqref{eq:scal-curv-formula-in-terms-of-momentum-profile} and \eqref{eq:trace-formula-in-trems-of-psi} the continuity path \eqref{eq:Chen-continuity-path-tcscK} will be the following ODE:
\begin{equation}\label{eq:ODE-for-Chen-continuity-path-on-ruled-surf}
(1-t)[(1+\tau)\psi]' = -(1+t) - t[(1+\tau)\phi]'' - \hat{S}_{t} (1 + \tau).
\end{equation}
The general solution of this ODE is given by
\begin{equation*}
(1-t)(1+\tau)\psi(\tau) = -(1+t)\tau - t((1+\tau)\phi' + \phi) - \frac{\hat{S}_{t}}{2}(\tau^2 + 2\tau) + C_1,
\end{equation*}
and putting $\tau =0$ and using boundary values \eqref{eq:boundary conditions for momentum profile} and \eqref{eq:boundary conditions for psi}, we get $C_1= t\phi'(0)=t$. Thus, with the value \eqref{value-of-hat S t-on-ruled-surf} of the constant $\hat{S}_t$, the unique solution of the above ODE \eqref{eq:ODE-for-Chen-continuity-path-on-ruled-surf} is given by
\begin{equation}\label{eq:unique sol for ODE of continuity method on ruled surface}
(1-t)(1+\tau)\psi(\tau) = t -(1+t)\tau - t((1+\tau)\phi' + \phi) - \frac{t(2 + b + ab) - a - b - ab}{a^2 + 2a}(\tau^2 + 2\tau).
\end{equation}
One can also check that the other boundary condition $\psi(a)=b$ holds.

\vspace*{1mm}
For $t=0$ in \eqref{eq:unique sol for ODE of continuity method on ruled surface} we get that $\psi$ satisfies \eqref{eq:uniq-sol-for-J-equ-ruled-surf} which is expected because at $t=0$ the continuity path \eqref{eq:Chen-continuity-path-tcscK} is the J-equation \eqref{eq:J-equ}, and we proved above that $\psi \in (0,b)$ on $(0,a)$ when $b\geq \frac{a^2}{2(1+a)}.$ While for $t=1$ in \eqref{eq:unique sol for ODE of continuity method on ruled surface}, we have
\begin{equation*}
[(1+\tau)\phi]' = 1-2\tau + \frac{a-2}{a^2+2a}(\tau^2 + 2\tau),
\end{equation*}
which implies $\phi$ is given by \eqref{eq:uniq-sol-of-ODE-for-cscK-on-ruled-surf} (using the boundary condition $\phi(0)=0$). But the other required boundary condition $\phi(a)=0$ is not satisfied because one can compute that $\phi(a) = -\frac{2a^2}{3(a+2)}$. It implies that there is no cscK metric in the K\"ahler class $\Omega_a$, and hence there is no path of K\"ahler metrics $\{\om_{\vp_t}\}_{t\in[0,1]}$ in the class $\Omega_a = 2\pi(\mathrm{C}+a D_\infty)$ which solves \eqref{eq:Chen-continuity-path-tcscK} for any K\"ahler metric $\chi$ in the class $\Omega_b = 2\pi(\mathrm{C}+b D_\infty)$.

\vspace*{1mm}
Now, suppose the momentum profile of $\om_{\vp_t}$ is given by $\phi(\tau):=(a\tau - \tau^2)/a$, where $\tau\in[0,a]$ (as in Theorem \ref{thm:exis-of-twisted-cscK-on-min-ruled-surfaces} or see Eq. \eqref{eq:specific-momentum-profile-for-omega}). Then, from \eqref{eq:value-of-1-plus-tau-times-phi-prime-for-special-momentum-profile} and \eqref{eq:unique sol for ODE of continuity method on ruled surface} we get
\begin{align*}
&(1-t)(1+\tau)\psi(\tau) \\ &= -(1+t)\tau -\frac{2t(a - 1) \tau}{a} + \frac{3t \tau^2}{a} - \frac{t(2 + b + ab) - a - b - ab}{a^2 + 2a}(\tau^2 + 2\tau) \\ 
&= \frac{t (3a + 4 -b -ab) + a+ b + ab}{a^2 + 2a} \tau^2 + \frac{-t(3a^2 + 4a + 2b + 2ab) - a^2 + 2b(1 + a)}{a^2 + 2a} \tau,
\end{align*}
and dividing both sides by $(1+\tau)$ we obtain
\begin{equation*}
\begin{split}
(1-t)\psi(\tau) &= \frac{t(3a + 4 -b -ab) + a + b + ab}{a^2 + 2a} \frac{\tau^2}{1 + \tau} \\ &\hspace*{1cm} + \frac{-t(3a^2 + 4a + 2b + 2ab) - a^2 + 2b (1 + a)}{a^2 + 2a} \frac{\tau}{1 + \tau}.
\end{split}
\end{equation*}
Now, differentiating the above equation with respect to $\tau$ we have
\begin{equation}\label{eq:first-derivative-of-psi-Chen-continuity-path}
\begin{split}
(1-t)\psi'(\tau) &= \frac{t(3a + 4 -b -ab) + a + b + ab}{a^2 + 2a} \frac{\tau^2 + 2\tau}{(1 + \tau)^2} \\ &\hspace*{1cm} + \frac{-t(3a^2 + 4a + 2b + 2ab) - a^2 + 2b (1 + a)}{a^2 + 2a} \frac{1}{(1 + \tau)^2}.
\end{split}
\end{equation}
Note that the coefficient of $\tau^2/(1+\tau)$ is positive since $t(3a + 4 -b -ab) + a + b + ab > 0$ for all $t\in[0,1]$. But the coefficient of $\tau/(1+\tau)$ may not be positive for all $t\in[0,1]$. In fact, for $b < \frac{a^2}{2(a+1)}$ the coefficient of $\tau/(1+\tau)$ is negative for all $t\in[0,1]$.

\vspace*{1mm}
Suppose $b \geq \frac{a^2}{2(a+1)}$. If $0 \leq t \leq m_0 := \frac{2b(a+1) - a^2}{3a^2 + 4a + 2b(a+1)}$, then both $\psi(\tau) > 0, \psi'(\tau) > 0$ for all $\tau\in(0,a)$ and hence we must have $\psi(\tau)\in(0,b)$ for all $\tau\in(0,a)$. Note that $0 < m_0 < 1$. Therefore, 
\begin{align*}
R (\Omega_a, \Omega_b) \geq \frac{2b(a+1) - a^2}{3a^2 + 4a + 2b(a+1)} 
\end{align*}
whenever $b \geq \frac{a^2}{2(a+1)}$. Moreover, $R (\Omega_a, \Omega_b) < 1$ since there is no cscK metric in $\Omega_a$. In particular, we have $\frac{a+2}{5a + 6} \leq R(\Omega_a, \Omega_a) < 1$ for all normalized K\"ahler class $\Omega_a = 2\pi(\mathrm{C}+a D_\infty)$, where $a>0$. This proves Proposition \ref{prop:bound-of-Chen-Cheng-invariant-for-tcscK}.

\vspace*{1mm}
Differentiating \eqref{eq:first-derivative-of-psi-Chen-continuity-path} once more with respect to $\tau$ we have
\begin{equation*}
(1-t)\psi''(\tau) = \frac{2(a+1)\left(t(3a + 4 + b) + a - b\right)}{(a^2 + 2a) (1 + \tau)^3}.
\end{equation*}
If $b < \frac{a^2}{2(a+1)}$, then $b < a$ and hence $\psi''(\tau) > 0$ for all $\tau\in(0,a)$. It implies $\psi'(\tau) > \psi'(0)$ or all $\tau\in(0,a)$ if $b < \frac{a^2}{2(a+1)}$. But note that
\begin{align*}
\psi'(0) = \frac{-t(3a^2 + 4a + 2b + 2ab) - a^2 + 2b + 2ab}{a^2 + 2a},
\end{align*}
which is negative for $b < \frac{a^2}{2(a+1)}$. So we must have $\psi'(\tau) < 0$ for all $\tau\in[0,a]$ and hence $0=\psi(0) > \psi(a) = b$, which is a contradiction. Thus, for $b < \frac{a^2}{2(a+1)}$ we see that the solution $\psi$ given in \eqref{eq:unique sol for ODE of continuity method on ruled surface} for the above choice of momentum profile does not give a K\"ahler metric $\chi$ which solves the continuity path \eqref{eq:Chen-continuity-path-tcscK}.

\vspace*{2mm}
%%%%%%%%%%%%%%%%%%%%%%%%%%%%%%%%%%%%
\section*{Acknowledgements}
The author would like to thank Ved Datar for helpful discussions and continuous encouragement.  This work was supported in part by an Institute Post Doctoral Fellowship (IPDF) from the Indian Institute of Technology (IIT), Bombay. The author would also like to thank Saikat Mazumdar for hosting him as an IPDF at IIT, Bombay. The author thanks Rodrigo Ristow Montes for showing interest in this paper and thanks Zakarias Sj\"ostr\"om Dyrefelt for pointing out the articles \cite{SDyre22-cscK-min-model} and \cite{JianShiSong19-cscK-min-model} related to twisted cscK metrics and for helpful comments and discussions.

\vspace*{2mm}
%%%%%%%%%%%%%%%%%%%%%%%%%%%%%%%%%%%%%%%%%%%%%%


\begin{thebibliography}{100}
		
\bibitem{Ban06-extended-Fut} S. Bando, {\em An obstruction for Chern class forms to be harmonic}, {\sc Kod. Math. J.} {\bf 29} (2006), 337--345.
		
\bibitem{BHPdeVen-bk} W.P. Barth, K. Hulek, C. Peters and A.V. de Ven, {\em Compact Complex Surfaces}, Ergebnisse der Mathematik und ihrer Grenzgebiete. 3. Folge (A Series of Modern Surveys in Mathematics), Springer, 2004, ISBN 978-3-540-00832-2.
		
\bibitem{Cal82} E. Calabi, {\em Extremal K\"ahler metrics}, {\sc Seminar on Differential Geom. Ann. of Math. Stud.} {\bf 16} (1982), 259--290.
		
\bibitem{gaoChen21-J-dHYM} G. Chen, {\em The J-equation and the supercritical deformed Hermitian-Yang-Mills equation}, {\sc Invent. Math.} {\bf 225} (2021), 529--602.
		
\bibitem{Chen00a-low-bd-Mab} X.X. Chen, {\em On the lower bound of the Mabuchi energy and its application}, {\sc Internat. Math. Res. Notices}, {\bf 2000}, no. 12, 607--623.
		
\bibitem{Chen-parabolic-flow} X.X. Chen, {\em A new parabolic flow in K\"ahler manifolds}, {\sc Comm. Anal. Geom.} {\bf 12} (2004), no. 4, 837--852.
		
\bibitem{Che18-cont-path-tcscK} X.X. Chen, {\em On the existence of constant scalar curvature K\"ahler metric: a new perspective}, {\sc Ann. Math. Qu\'{e}.} {\bf 42} (2018), no. 2, 169--189.
		
\bibitem{ChenCheng21-cscK-exis} X.X. Chen and J. Cheng,  {\em On the constant scalar curvature Kähler metrics (II)-Existence results}, {\sc J. Amer. Math. Soc.} {\bf 34} (2021), 937--1009.
		
\bibitem{ColSze17-J-flow-toric} T. Collins and G. Sz\'ekelyhidi, {\em Convergence of the J-flow on toric manifolds}, {\sc J. Differential. Geom.} {\bf 107} (2017), no. 1, 47--81.
		
\bibitem{DMS-min-slopes-cmpx-Hess-eq} V.V. Datar, R. Mete and J. Song, {\em Minimal slopes and bubbling for complex Hessian equations}, \url{https://arxiv.org/abs/2312.03370}
		
\bibitem{Datar-Pingali Coup-cscK} V.V. Datar and V.P. Pingali, {\em On coupled constant scalar curvature metrics}, {\sc J. Symplectic Geom.} {\bf 18} (2020), no. 4, 961--994.
		
\bibitem{DatPin2021-gCMA-proj-mfds} V.V. Datar and V.P. Pingali, {\em A numerical criterion for generalized Monge-Ampere equations on projective manifolds}, {\sc Geom. Funct. Anal.} {\bf 31} (2021), no. 4, 767--814.

\bibitem{Delcroix-Hultgren2021} T. Delcroix and J. Hultgren, {\em Coupled Monge-Amp\`{e}re equations on Fano horosymmetric manifolds}, {\sc J. Math. Pures Appl.} (9) {\bf 153} (2021), 281--315.

\bibitem{Der15-tcscK} R. Dervan, {\em Uniform stability of twisted constant scalar curvature K\"ahler metrics}, {\sc Int. Math. Res. Not. IMRN}, {\bf 15} (2016), 4728--4783.
		
\bibitem{Don99-j-eq} S.K. Donaldson, {\em Moment maps and diffeomorphisms}, {\sc Asian J. Math.} {\bf 3} (1999), no. 1, 1--16.
		
\bibitem{FL-Calabi ansatz-2013} H. Fang and M. Lai, {\em Convergence of general inverse $\sigma_k$-flow on K\"ahler manifolds with Calabi ansatz}, {\sc Trans. Amer. Math. Soc.} {\bf 365} (2013), no. 12, 6543–-6567.
		
\bibitem{FLSW14-J-flow-on-Kah-surf-boundary-case} H. Fang, M. Lai, J. Song and B. Weinkove, {\em The J-flow on K\"ahler surfaces: a boundary case}, {\sc Anal. PDE}, {\bf 7} (2014), no. 1, 215--226.
		
\bibitem{Fuj92} A. Fujiki, {\em Remarks on Extremal K\"ahler Metrics on Ruled Manifolds}, {\sc Nagoya Math. J.} {\bf 126} (1992), 89--101.
		
\bibitem{Fut83} A. Futaki, {\em An obstruction to the existence of Einstein K\"ahler metrics}, {\sc Invent. Math.} {\bf 73} (1983) 437--443.
		
\bibitem{Fut83b} A. Futaki, {\em On compact K\"ahler manifolds of constant scalar curvatures}, {\sc Proc. Japan Acad. Ser. A Math. Sci.} {\bf 59} (1983), no. 8, 401--402.
		
\bibitem{Fut88-bk} A. Futaki, {\em K\"ahler-Einstein metrics and integral invariants}, {\sc Lecture Notes in Mathematics, Springer}, 1988.
		
\bibitem{Hashimoto19-tcscK} Y. Hashimoto, {\em Existence of twisted constant scalar curvature K\"ahler metrics with a large twist}, {\sc Math. Z.} {\bf 292} (2019), 791--803.
		
\bibitem{Hatcher-bk} A. Hatcher, {\em Algebraic Topology}, {\sc Cambridge University Press}, 2002, ISBN 0521795400, 9780521795401.
		
\bibitem{Hultgren-Witt Nystrom2018} J. Hultgren and D. Witt Nystr\"om, {\em Coupled K\"ahler-Einstein metrics}, {\sc Int. Math. Res. Not. IMRN}, {\bf 2019}, no. 21, 6765--6796.
		
\bibitem{Hultgren2019} J. Hultgren, {\em Coupled K\"ahler-Ricci solitons on toric Fano manifolds}, {\sc Anal. PDE}, {\bf 12} (2019), no. 8, 2067--2094.	
		
\bibitem{HwangSinger02} A.D. Hwang and M.A. Singer, {\em A momentum construction for circle invariant K\"ahler metrics}, {\sc Trans. Amer. Math. Soc.} {\bf 354} (2002), no. 6, 2285--2325.
		
\bibitem{JianShiSong19-cscK-min-model} W. Jian, Y. Shi and J. Song, {A remark on constant scalar curvature K\"ahler metrics on minimal models}, {\sc Proc. Amer. Math. Soc.} {\bf 147} (2019), no. 8, 3507--3513.
		
\bibitem{LeBrun-Singer93} C. LeBrun and M.A. Singer, {\em Existence and deformation theory for scalar-flat K\"ahler metrics on compact complex surfaces}, {\sc Invent. Math.} {\bf 112} (1993), no. 2, 273--313.
		
\bibitem{LejSze15-J-flow-stab} M. Lejmi and G. Sz\'ekelyhidi, {\em J-flow and stability}, {\sc Adv. Math.} {\bf 274} (2015), 404--431.
		
\bibitem{Mab86} T. Mabuchi, {\em K-energy maps integrating Futaki invariant}, {\sc Tohoku Math. J.} {\bf 38} (1986), 575--593.
		
\bibitem{SDyre22-cscK-min-model} Z. Sj\"{o}str\"{o}m Dyrefelt, {\em Existence of cscK metrics on smooth minimal models}, {\sc Ann. Sc. Norm. Super. Pisa Cl. Sci. (5)} {\bf 23} (2022), no. 1, 223--232.
		
\bibitem{SDyre20-j-func} Z. Sj\"{o}str\"{o}m Dyrefelt, {\em Optimal lower bounds for Donaldson’s J-functional}, {\sc Adv. Math.} {\bf 374} (2020), 107271, 37 pp.
		
\bibitem{Song20-NM-crit-J-eq} J. Song, {\em Nakai-Moishezon criterion for complex Hessian equations}, \url{https://arxiv.org/abs/2012.07956}.
		
\bibitem{SW08-J-flow-conv} J. Song and B. Weinkove, {\em On the convergence and singularities of the J-flow with applications to the Mabuchi energy}, {\sc Comm. Pure Appl. Math.} {\bf 61} (2008), no. 2, 210–-229.
		
\bibitem{Sto09-tcscK} J. Stoppa, {\em Twisted constant scalar curvature K\"ahler metrics and K\"ahler slope stability}, {\sc J. Differential Geom.} {\bf 83} (2009), no. 3, 663--691.
		
\bibitem{Sze-bk} G. Sz\'ekelyhidi, {\em An Introduction to Extremal K\"ahler Metrics}, {\sc Grad. Stud. Math., vol. {\bf 152}, Amer. Math. Soc.} (2014), ISBN 978-1-4704-1047-6.
		
\bibitem{To23-degenrate-tJ-flow} T.D. T\^{o}, {\em Degenerate J-flow on compact K\"ahler manifolds}, {\sc Math. Z.} {\bf 303} (2023), no. 97, 1--24.
		
\bibitem{Ton98-extrem} C.W. Tonnesen-Friedman, {\em Extremal K\"ahler Metrics on Minimal Ruled Surfaces}, {\sc J. Reine Angew. Math.} {\bf 502} (1998), 175--197. 
		
\bibitem{Weinkove04-convergence-surface} B. Weinkove, {\em Convergence of the J-flow on K\"ahler surfaces}, {\sc Comm. Anal. Geom.} {\bf 12} (2004), no. 4, 949--965.
		
\bibitem{Weinkove06-convergence-higher-dim} B. Weinkove, {\em On the J-flow in higher dimensions and the lower boundedness of the Mabuchi energy}, {\sc J. Differential Geom.} {\bf 73} (2006), no. 2, 351--358.
		
\bibitem{Zen19-tcscK} Y. Zeng, {\em Deformations from a given K\"ahler metric to a twisted cscK metric},  {\sc Asian J. Math.} {\bf 23} (2019), no. 6, 985--1000.
		
\end{thebibliography}
\end{document}